\documentclass[11pt]{article}
\usepackage{graphicx}
\usepackage{amsmath}
\usepackage{amsfonts}
\usepackage{amssymb}

\newtheorem{lemma}{Lemma}
\newtheorem{theorem}{Theorem}

\newtheorem{definition}{Definition}

\def\QED{~\rule[-1pt]{5pt}{5pt}\par\medskip}
\newenvironment{proof}{{\bf Proof: \ }}{ \hfill \QED}

\newcommand{\be}{\begin{equation}}
\newcommand{\ee}{\end{equation}}
\newcommand{\bea}{\begin{eqnarray}}
\newcommand{\eea}{\end{eqnarray}}
\newcommand{\beas}{\begin{eqnarray*}}
\newcommand{\eeas}{\end{eqnarray*}}

\newcommand{\ones} {{\bf 1}}

\newcommand{\ba}{\begin{array}}
\newcommand{\ea}{\end{array}}

\newcommand{\Real}{\mathbb{R}^}

\newcommand{\ac}{{\bf a}}
\newcommand{\lap}{\mathcal L}
\newcommand{\bE}{{\bf E}}
\newcommand{\bV}{{\bf V}}
\newcommand{\bG}{{\bf G}}
\newcommand{\bd}{{\bf d}}
\newcommand{\bfe}{{\bf e}}

\newcommand{\zero} {{\bf 0}}

\newcommand{\bbm}{\begin{bmatrix}}
\newcommand{\ebm}{\end{bmatrix}}

\title{\LARGE Spectrum of  Laplacians for Graphs with Self-Loops}
\author{\Large Beh\c cet\ A\c c\i kme\c se \\
\small Department of Aerospace Engineering and Engineering Mechanics \\
\small The University of Texas at Austin, USA}

\begin{document}
\maketitle
\begin{abstract}
This note introduces a result on the
location of eigenvalues, i.e., the spectrum, of the Laplacian for a
family of undirected graphs with self-loops. We extend on the known results for the spectrum of undirected graphs without self-loops or multiple edges. For this purpose, we introduce a new concept of pseudo-connected graphs and apply a lifting of the graph with self-loops to a graph without self-loops, which is then used to quantify the spectrum of the Laplacian for the graph with  self-loops.
\end{abstract}

\section{Introduction}
Graph theory has proven to be an extremely useful mathematical framework for many emerging engineering applications \cite{mehran_book,morse_03,Boyd07,mohar_97,murray_saber_04,ren_tac_05}.
This note introduces a  result, Theorem \ref{thm:bnd_loopG}, on the
location of eigenvalues, i.e., the spectrum, of the Laplacian for a
family of undirected graphs with self-loops. We extend on the known results for the spectrum of undirected graphs without self-loops  or multiple edges \cite{fiedler73,horn_99,RKChung_94}. For this purpose, we introduce a new concept of pseudo-connected graphs and apply a lifting of the graph with self-loops to a graph without self-loops, which is then used to quantify the spectrum of the Laplacian for the graph with  self-loops.

The primary motivation for this result is to analyze interconnected control systems that emerge from new engineering applications such as control of multiple vehicles or spacecraft \cite{roysmith_09,behcet_07ffs,acikmese2012markov,behcet_14markov,morse_03}. 
This result first appeared \cite{behcet_14decentralized}, where it proved key to extend the  stability proofs for observers \cite{behcet_05observer,behcet_11observers} to  decentralized observers. We believe that the result can also be useful in designing decentralized control \cite{Speyer79,Mesbahi08,Willsky82,Mutambara98,mehran_cdc05} and optimization algorithms.

\paragraph*{Notation:}
%
The following is a partial list of notation  (see the Appendix for the graph theoretic notation):
$\Real n$ is the $n$ dimensional real vector space;
$\|\cdot\|$ is the vector 2-norm;
$I$ is the identity matrix  and $I_m$ is the identity matrix in $\Real {m \times m}$; $\ones_m$ is a  vector of ones in $\Real m$; $\bfe_i$ is a vector with its $i$th entry $+1$ and the others zeros; $\sigma(A)$ is the set of eigenvalues of  $A$;  $\sigma_+(S)$ are the positive eigenvalues of  $S\!=\!S^T$;
$\rho(A)$ is the spectral radius of $A$.
%

\section{Main Result on Laplacians for Graphs with Self-Loops}
This section first summarizes  some known facts about graph theory, and then introduces  the main result of this note  on graphs with self-loops. For graphs with self-loops, we will introduce the concept of pseudo connectedness, which is useful in our developments.
\\
Let $\bG=(\bV,\bE)$ represents a finite graph with a set of vertices $\bV$ and edges $\bE$ with $(i,j) \in \bE$ denoting  an edge between the vertices $i$ and $j$.
$\lap (\bG)$ is the Laplacian matrix for the graph $\bG$; $\ac (\bG)$ is the algebraic connectivity of the graph $G$, which is the second smallest eigenvalue of $\lap (\bG)$.
%
$E$ is the vertex-edge adjacency matrix.
Each row of the vertex-edge adjacency matrix describes an edge between two vertices with
 entries corresponding to these vertices are $+1$ and $-1$ (it does not matter which entry is $+$ or $-$) and the rest of the entries are zeros. Note that if the edge described by a row is  a self-loop then there is only one non-zero entry with $+1$. Hence   a row  of $E_{i,k}$,  denoted by $\pi$, defining an edge between $p$'th and $q$'th vertices of the graph  has  its $j$th entry of $\pi_j$ as follows 
$$
 \pi_{j} =   \left\{\begin{array}{cc}1 &  j =p \\[-3pt]  -1^{(q-p)} &  j=q \\[-3pt]  0 & {\rm otherwise } \end{array}\right. .
$$
$\mathcal{A}$ is the adjacency matrix, and $\mathcal{D}$ is the diagonal matrix of node in-degrees  for $\bG$, then the following gives a relationship to compute the Laplacian matrix
\be
\lap(\bG) = E^T E = \mathcal{D} - \mathcal{A}.
\ee
The following relationships are well known in the literature  \cite{fiedler73} and \cite{horn_99} for a connected undirected graph $\bG$  with $N$ vertices and  without any {\em self-loops or multiple edges}
\begin{eqnarray}
\qquad \qquad \ac (\bG) &\geq & 2(1-\cos(\pi/N))
\label{eq:fiedler_bnd} \\
 \qquad  \qquad 2\bd(\bG) & \geq & \max(\sigma(\lap (\bG))),
\label{eq:fiedler_bnd2}
\end{eqnarray}
where $\bd(\bG)$ is the maximum in-degree of $\bG$. Indeed the inequality (\ref{eq:fiedler_bnd2}) is valid for any undirected graph without self-loops or multiple edges whether they are connected or not. Also, due to the connectedness of the graph, the minimum eigenvalue of the Laplacian matrix is $0$ with algebraic multiplicity of 1 and the eigenvector of $\ones$.
Next we characterize the location  of the Laplacian eigenvalues for a connected  undirected graph  $\bG$ with self-loops.  Having a  self-loop does not change whether a graph is connected or not, that is,  a graph with self-loops is connected if and only if the same graph with the self-loops removed is connected. Furthermore, we define the Laplacian of an undirected graph with at least one  self-loop as
\be
\mathcal{L}(\bG) = \mathcal{L}(\bG^o) + \sum_{(i,i) \in \bE} \bfe_i \bfe_i^T
\label{eq:L4loops}
\ee
where  $\bG^o$ is the largest subgraph of $\bG$ with the self-loops removed, and
\be
\mathcal{L}(\bG^o) = \sum_{(i,j)\in \bE,\, i \neq j} (\bfe_i-\bfe_j)(\bfe_i-\bfe_j)^T.
\label{eq:LwoutLoops}
\ee
The following definition introduces the concept of the pseudo-connected graphs, which is our fourth contribution.
\begin{definition} \label{def:pcon}
An undirected graph $\bG(\bV,\bE)$ without multiple edges  is {\em pseudo-connected} if every vertex is connected to itself and/or to another vertex  and if every connected subgraph of $\bG$ has at least one vertex with a self-loop.
\end{definition}
%

Next  we develop useful results on the eigenvalues of  undirected graphs with self-loops, which are  instrumental in the  stability analysis of the decentralized observer. We refer  to \cite{mehran_book} for a graph theoretic view  of multi-agent networks. 
\begin{lemma} \label{lem:Lpd}
The Laplacian of a pseudo-connected graph is positive definite.
\end{lemma}
\begin{proof}
A pseudo-connected graph can be partitioned into subgraphs that are connected with at least one self-loop in each subgraph. Note that some of these subgraphs can have a single vertex that has a self-loop. Clearly each subgraph with a single vertex and a self-loop has Laplacian $1$. If we can also show that the connected subgraphs that have multiple vertices with at least one self-loop have positive definite Laplacians, then the  Laplacian of the overall graph will also be positive definite.  This will conclude the proof.
To do that we prove that a connected graph $\bG$ with at least one self-loop has a positive definite Laplacian. Let $\bG^o$ be the connected graph formed by removing the self-loops from $\bG$.
Any vector $v \! \neq \! \zero$, which  can be expressed as  $v \!=\!w \!+\! \zeta \ones$  where $w^T \ones \!=\!0$, and either or both $w \! \neq \! \zero$ and $\zeta\!\neq \!0$. 
Then, by using (\ref{eq:LwoutLoops}) $\mathcal{L}(\bG)  = \mathcal{L} (\bG^o) +Q_o$  where $Q_o\!:=\! \sum_{i=1}^q e_ie_i^T $ and $q$ is the number of self-loops and having $\ones^T Q_o \ones = q$,
$$
v^T \mathcal{L}(\bG) v \!=\! w^T \mathcal{L} (\bG^o) w  \!+\! w^T Q_o w \!+\! 2\zeta w^T Q_o \ones \!+\! q\xi^2 \geq 0.  
$$
If $w \!\neq \!\zero$, $w^T \mathcal{L} (\bG^o) w >0$ (due to connectedness of $\bG^o$), we have $v^T \mathcal{L}(\bG) v >0$. Next, if $w\!=\!0$ and $\zeta \!\neq\! 0$, then $v^T \mathcal{L}(\bG) v \!=\!  q \zeta^2 >0 $. Consequently $\mathcal{L}(\bG)\!=\! \mathcal{L}(\bG)^T \!\! >\! 0$, where $q$ is the number of self-loops.
\end{proof}
Next we introduce the concept of  {\em lifted graph} to characterize the eigenvalues of the Laplacian of a graph with self-loops.
\begin{definition}
Given an undirected  graph $\bG(\bE,\bV)$ with $N$ vertices and with at least one self-loop, its lifted graph $\hat{\bG}(\hat{\bE},\hat{\bV})$ is a graph with $2N+1$ vertices and with no self-loops such that (Figure \ref{fig:lifting}): For every vertex $i$ in $\bG$ there are vertices $i$ and $i\!+\!N\!+\!1$ in $\hat{\bG}$, $i=1,...,N$, and also a {\em middle vertex} $N\!+\!1$ with  the following  edges
\beas
 (i,j) \! \in \! \bE \, \Rightarrow (i,j) \in \hat{\bE} \   {\rm and}  \ (i+N+1,j+N+1) \! \in \! \hat{\bE}&& \\
 (i,i) \! \in \! \bE \, \Rightarrow (i,N+1) \in \hat{\bE}  \ {\rm and}  \ (N+1,i+N+1) \! \in \! \hat{\bE}.&&
\eeas
\end{definition}
\begin{figure}[ht]
\centering \fbox{\includegraphics[width=8.4cm,height=5.9cm]{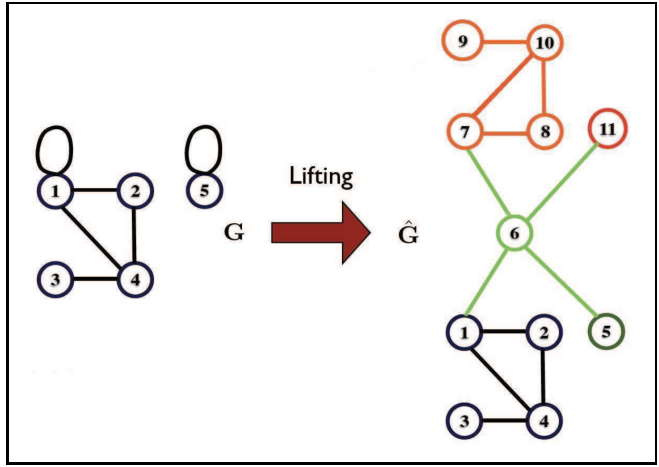}}
\caption{\small{Lifted graph of a pseudo-connected graph with self-loops.}}
\label{fig:lifting}
\end{figure}
The following theorem is the main result of this section on the eigenvalues of the Laplacians of pseudo-connected graphs.
\begin{theorem} \label{thm:bnd_loopG}
For a finite undirected graph, $\bG$, with self-loops but without multiple-edges:
\be \label{eq:eigGg}
\sigma\left(\mathcal{L}(\bG)\right) \subseteq \sigma \left(\mathcal{L}(\hat{\bG})\right) \cap [0,2\bd(\bG^o)\!+\! 1],
\ee
where $\bG^o(\bV,\bE^o)$ is  a subgraph of $\bG(\bV,\bE)$ where $\bE^o \subset \bE$ and $\bE^o$ contains all the edges of $\bE$ that are not self-loops.
Particularly, if $\bG$ is a pseudo-connected graph, then
\be \label{eq:eigG}
\sigma\left(\mathcal{L}(\bG)\right) \subseteq \sigma_+ \left(\mathcal{L}(\hat{\bG})\right) \cap [0,2\bd(\bG^o)\!+\!1].
\ee
\end{theorem}
\begin{proof}
Consider the edge-vertex adjacency matrix   $E^o$ for $\bG^o$. We have the following relationship for the vertex adjacency matrices of $\bG$ and $\hat{\bG}$, $E$ and $\hat{E}$,  in terms of   $E^o$ 
$$
\hat{E} = \left[\begin{array}{ccc}
E^o & 0 & 0 \\
S & \ones & 0 \\0 & \ones & S \\0 & 0 & E^o \end{array}\right], \quad  E = \left[\begin{array}{c} E^o \\ S \end{array}\right]
$$
where the  matrix $S$ has  entries of $+1$ or $0$.  This implies that
$$
\mathcal{L}(\hat{G}) = \left[\begin{array}{ccc}
{E^o}^TE^o+S^TS  & S^T \ones & 0 \\ \ones^T S  & 2 N & \ones^T S \\0 & \ \, S^T \ones  & {E^o}^TE^o + S^TS\end{array}\right]
$$
and $\mathcal{L}(\bG) = {E^o}^TE^o + S^T S$.
Now suppose that $\psi \in \sigma\left(\mathcal{L}(\bG)\right)$ with the corresponding eigenvector $v$.
Then
$$
\mathcal{L}(\hat{\bG}) \left[\begin{array}{c}v \\0 \\ -v \end{array}\right] =  \left[\begin{array}{c}
\mathcal{L}(\bG) v \\
0 \\
-\mathcal{L}(\bG) v\end{array}\right] = \psi \left[\begin{array}{c}v \\0 \\ -v \end{array}\right].
$$
Consequently $\psi \in \sigma \left(\mathcal{L}(\hat{\bG}) \right)$ too.
Next note that $0 \leq S^TS \leq I$, which implies that $\mathcal{L}(\bG) \leq \mathcal{L}(\bG^o) + I $.  This implies that
\be
\max(\sigma \left( \mathcal{L}(\bG) \right)) \leq \max(\sigma (\mathcal{L}(\bG^o)))\! +\! 1\! \leq\! 2 \bd(\bG^o)+1
\ee
which follows from (\ref{eq:fiedler_bnd2}). This proves the relationship given by (\ref{eq:eigGg}). Now by using Lemma \ref{lem:Lpd}, the relationship given by (\ref{eq:eigG}) directly follows from (\ref{eq:eigGg}).
\end{proof}

\bibliographystyle{unsrt}
\bibliography{gr_selfloop}
\end{document}